\documentclass{amsart}
\usepackage[utf8]{inputenc}

\usepackage[dvipsnames]{xcolor}

\usepackage{xypic}
\usepackage{tikz}
\usetikzlibrary{trees}
\usetikzlibrary{arrows.meta, decorations.pathmorphing, decorations.shapes, decorations.pathreplacing, calc}

\usepackage{anyfontsize}

\usepackage{amsfonts}
\usepackage{amssymb}
\usepackage{amsmath}
\usepackage{amsthm}

\usepackage{thmtools,thm-restate}

\usepackage{ulem}

\usepackage{biblatex} 
\theoremstyle{plain}
\newtheorem{Thm}{Theorem}
\newtheorem{Cor}{Corollary}

\newtheorem{Lem}[Thm]{Lemma}

\newtheorem{Conj}[Thm]{Conjecture}

\theoremstyle{definition}
\newtheorem{Def}[Thm]{Definition}

\newcommand{\NN}{\mathbb{N}}

\newcommand{\restr}{\upharpoonright}

\newcommand{\calP}{\mathcal{P}}
\newcommand{\calF}{\mathcal{F}}

\DeclareMathOperator{\stab}{Stab}
\DeclareMathOperator{\aut}{Aut}
\DeclareMathOperator{\sym}{Sym}

\DeclareMathOperator{\pred}{Pred}
\DeclareMathOperator{\cay}{Cay}

\begin{document}
\title{Frucht's theorem without choice}
\author{Brian Pinsky}
\begin{abstract}
    Frucht's theorem is the statement that ``every group is the automorphism group of a graph".  This was shown over ZFC by Sabidussi \cite{sabidussi} and deGroot\cite{degroot}, by induction using a well ordered generating set for the group. Sabidussi's proof is easily modified to use induction on the rank of a generating set, and thus holds over ZF.

    We show that Frucht's theorem is independent of ZFA set theory (ZF with atoms), by showing it fails in several common permutation models.  We also present some permutation models where Frucht's theorem holds, even when AC fails.

    As a corollary, we show that a stronger version of Frucht's theorem due to Babai in \cite{BABAI197826} can fail without choice.
\end{abstract}
\maketitle

\section{Introduction}

Frucht's theorem says every group is isomorphic to the automorphism group of a simple graph (with no loops, multiple edges or directed edges).  This was proved by Frucht \cite{frucht} in 1939 for finite groups, and generalized independently by deGroot \cite{degroot} and Sabidussi \cite{sabidussi} to infinite groups in 1959-1960.  All the proofs use the same basic idea; given a group $\Gamma$, take an (edge-)colored directed Cayley graph for $\Gamma$, with one color for each generator.  The automorphism group of this colored Cayley graph is $\Gamma$, but it is not a simple graph.
The constructions in \cite{degroot} and \cite{sabidussi} then add additional structure to encode the colors, as we will summarize in section \ref{zf}.

Frucht's theorem for permutation models is more interesting. In section \ref{counterexamples}, we show Frucht's theorem fails in several permutation models (e.g. the basic Mostowski model).  We show this by examining ``internal" automorphisms of these models; that is, automorphisms of the atoms which are themselves symmetric.  Frucht's theorem will fail in any model with ``enough" internal automorphisms.
Theorem \ref{simplified_no_frucht} gives a general criterion for a Frucht's theorem to fail for a particular group.
Since Frucht's theorem is true in ZF, its failure in these ZFA models can't be transfered to a symmetric model.
This is similar to the results in \cite[chapter 9]{jech}, but differs in that Frucht's theorem does not imply AC over ZF.

On the other hand, in Section \ref{non-counterexamples},  we show Frucht's theorem holds in some permutation models (e.g. the ordered Mostowski model), even when choice fails.  
The atoms of these models have some structure which becomes rigid in the permutation model, so they have no internal automorphisms.  In these cases, we can code the rigid structure of the atoms into rigid graphs, and adapt the proof of Frucht's theorem from section \ref{zf}.

In section \ref{babai}, we look at Babai's theorem \cite{BABAI197826}: if $\Gamma$ is infinite, then $\Gamma$ is the transitive automorphism group of some directed graph (and also an undirected graph with 3 vertex orbits).  By applying the Jech-Sochor embedding to a ZFA model where Frucht's theorem fails, we show Babai's theorem can fail in ZF.

\subsection{Definitions}
$V$ is the universe of all sets, and $V_\alpha$ is the set of all sets of rank $\alpha$.  In particular, in ZFA, $V_0$ is the set of atoms, and $V_0=\varnothing$ in ZF.

Inside $V$, a \textbf{graph} (or ``simple graph") $G$ is a set $\langle V(G),E(G)\rangle$ where $V(G)$ is any set and $E(G)\subseteq V^2$ is a symmetric, irreflexive relation.  A \textbf{directed graph} only requires $E(G)$ be irreflexive.  A \textbf{colored (directed) graph} with a set $S$ of colors is a (directed) graph with an $S$-valued edge-coloring, i.e. a function from $E(G)$ to $S$.
A \textbf{pointed graph} is a graph with a distinguished vertex.

A \textbf{homomorphism} of graphs is a function on vertices that preserves both edges and non-edges.  Homomorphisms must also preserve edge color for colored graphs, and the distinguished vertex for pointed graphs.
A \textbf{rigid graph} is one with no non-trivial automorphisms.
A graph is \textbf{connected} if any two verticies are connected by a path, i.e. a finite sequence of adjacent edges.  A \textbf{RCPG} is a rigid, connected, pointed graph.

Finally, if $S$ is a generating set for a group $\Gamma$, the corresponding \textbf{colored directed Cayley graph} $\widetilde{\cay}(\Gamma,S)$ is the colored drected graph with color set $S$, vertex set $\Gamma$, and whose $s$-colored edges are $\{\langle g,gs\rangle | g\in \Gamma \}$.  $\cay(\Gamma,S)$ is the same graph without the edge colors.

\section{Frucht's theorem in ZF}\label{zf}

\begin{Thm}[Frucht's theorem]
    Every group is the automorphism group of a graph
\end{Thm}
It is easy to show every group $\Gamma$ is the automorphism group of the corresponding colored directed Cayley graph  $\widetilde{\cay}(\Gamma,\Gamma)$.  So, to show Frucht's theorem, we will give two lemmas to replace a colored directed graph with a simple graph.

First, suppose we are given a colored directed graph and several rigid, connected, pointed graphs (RCPGs) to use as labels.  By attaching RCPGs to each edge indicating its color and direction, we can produce a simple graph with the same automorphism group, as pictured below (using sticks of length 2 and 3 as labels):
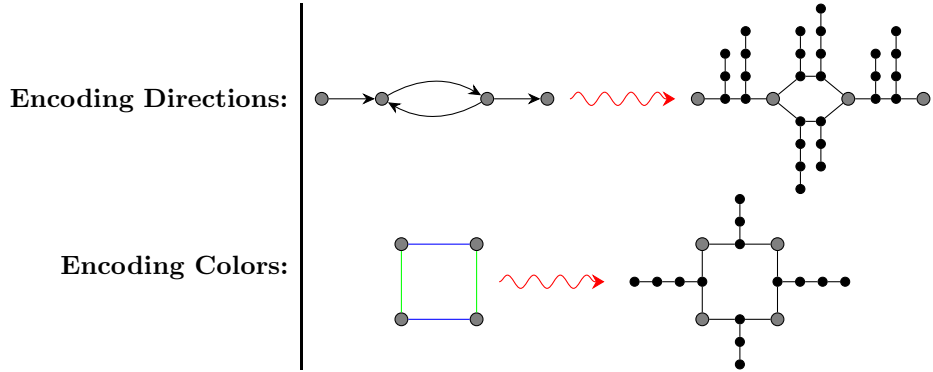
\begin{center}
    \begin{tabular}{r|c}
        \textbf{Encoding Directions:}    &  
        \begin{tikzpicture}[
                baseline = {([yshift = -.7ex]current bounding box.center)},
                 >={Stealth[width=1.5mm]},
                old/.style ={circle, draw, fill=gray, inner sep=0pt, minimum size=1.7mm},
                new/.style ={circle, draw, fill, inner sep=0pt, minimum size=1.2mm},
                label/.style ={ }
        ]
        \usetikzlibrary{calc, decorations.pathmorphing}
        \begin{scope}[xshift=-3cm]
            \node[old] (a) at (0,0) {};
            \node[old] (b) at (.8,0) {}
                edge[<-] (a);
            \node[old] (c) at (2.2,0) {}
                edge[->, bend left = 30] (b)
                edge[<-, bend right = 30] (b);
            \node[old] (d) at (3,0) {}
                edge[<-] (c);
        \end{scope}

        \draw[-{>[width=2mm]}, decorate, decoration=snake, color=red] (.3, 0) -- (1.7, 0);

        \newcommand{\drawstick}[5]{
            \foreach \x [var = \prev, evaluate =\prev using int(\x-1)]in {1,...,#2}{
                \node[new] (c#1_\x) at ($(#1)+\x*(#3:#4)+(#3:#5)$) {}
                    \ifnum \x > 1
                        edge (c#1_\prev) to (c#1_\x)
                    \fi;
            }
            \draw (c#1_1) to (#1);
        }
        \newcommand{\drawDirEdge}[4]{
            \path (#1)--(#2)
                node[new, pos=1/3, yshift = #3] (int1) {}
                node[new, pos=2/3, yshift = #3] (int2) {};
            \drawstick{int1}{2}{#4}{.3}{0};
            \drawstick{int2}{3}{#4}{.3}{0};
            \draw (#1) -- (int1) -- (int2) -- (#2);
        }
        
        \begin{scope}[xshift = 2cm]
            \node[old] (a) at (0,0) {};
            \node[old] (b) at (1,0) {};
            \node[old] (c) at (2,0) {};
            \node[old] (d) at (3,0) {};
            \drawDirEdge{a}{b}{0}{90}
            \drawDirEdge{b}{c}{.3cm}{90}
            \drawDirEdge{c}{b}{-.3cm}{270}
            \drawDirEdge{c}{d}{0}{90}
        \end{scope}

        \end{tikzpicture}\\
        
        \textbf{Encoding Colors:} & 
        \begin{tikzpicture}[
                baseline = {([yshift = .7ex]current bounding box.center)},
                 >={Stealth[width=1.5mm]},
                old/.style ={circle, draw, fill=gray, inner sep=0pt, minimum size=1.7mm},
                new/.style ={circle, draw, fill, inner sep=0pt, minimum size=1.2mm},
                label/.style ={ }
        ]
        \usetikzlibrary{calc, decorations.pathmorphing}
        \begin{scope}[xshift=-2cm]
            \node[old] (a) at (0,0) {};
            \node[old] (b) at (1,0) {}
                edge[blue] (a);
            \node[old] (c) at (0,1) {}
                edge[green] (a);
            \node[old] (d) at (1,1) {}
                edge[green] (b)
                edge[blue] (c);
        \end{scope}

        \draw[-{>[width=2mm]}, decorate, decoration=snake, color=red] (-.7, .5) -- (.7, .5);

        \newcommand{\drawstick}[5]{
            \foreach \x [var = \prev, evaluate =\prev using int(\x-1)]in {1,...,#2}{
                \node[new] (c#1_\x) at ($(#1)+\x*(#3:#4)+(#3:#5)$) {}
                    \ifnum \x > 1
                        edge (c#1_\prev) to (c#1_\x)
                    \fi;
            }
            \draw (c#1_1) to (#1);
        }
        \newcommand{\drawDirEdge}[4]{
            \path (#1)--(#2)
                node[new, pos=.5] (int#1#2) {};
            \drawstick{int#1#2}{#3}{#4}{.3}{0};
            \draw (#1) -- (int#1#2) -- (#2);
        }
        
        \begin{scope}[xshift = 2cm]
            \node[old] (a) at (0,0) {};
            \node[old] (b) at (1,0) {};
            \node[old] (c) at (0,1) {};
            \node[old] (d) at (1,1) {};
            \drawDirEdge{a}{b}{2}{270}
            \drawDirEdge{b}{d}{3}{0}
            \drawDirEdge{a}{c}{3}{180}
            \drawDirEdge{c}{d}{2}{90}
        \end{scope}

        \end{tikzpicture}
    \end{tabular}
\end{center}
Details of this construction can be found in \cite{details}.  This construction requires one rigid connected pointed graph for each color, giving the following:
\begin{Lem}\label{enough_labels_implies_fructh}
    Given an $S$ colored directed graph $G$ and an equinumerous with $S$ set of pairwise non-isomorphic RCPGs, there is a simple graph $G'$ with the same automorphism group as $G$
\end{Lem}

Sabidussi's construction in \cite{4} works by induction on rank to construct a RCPG for each set.  He assumes AC, so he only cares about the graphs for ordinals. However, if we isolate his inductive step, he proves the following:
\begin{Lem}\label{inductively_building_labels}
    Suppose there is a function $R$ from a set $S$ into pairwise non-isomorphic RCPGs.  Then $R$ can be extended to a function from $\calP(S)$ into pairwise non-isomorphic RCPGs.
\end{Lem}
We will sketch the construction, but refer to \cite{details} or \cite{sabidussi} for details.  Pick $x\in \calP(S)\setminus S$.  We start with 2 complete graphs on $S$, attached as shown.  On the bottom copy, using $R$, we attach a RCPG to each vertex, making it rigid.  On the top copy, we attach degree 1 verticies for each element of $x$, to encode $x$.  Finally, we attach one distinguished vertex, $t$, to each top vertex.
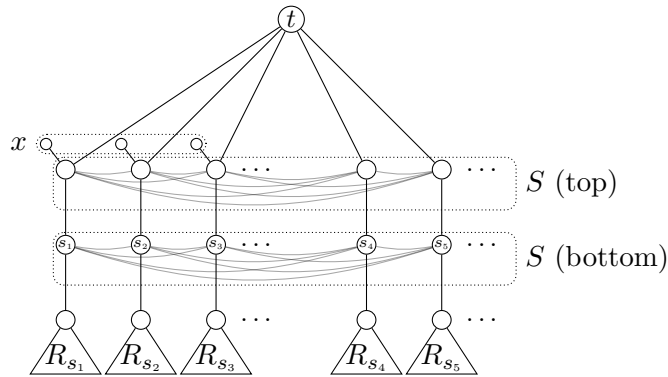
\begin{center}

\usetikzlibrary{arrows.meta, decorations.pathmorphing, shapes, fit, positioning, calc, through}
\begin{tikzpicture}[
		>={Stealth[width=1.5mm]},
		vertex/.style ={circle, draw, inner sep=1pt, minimum size=2.5mm},
		label/.style ={font=\large , opacity=1 },
		region/.style = {rounded corners, densely dotted, opacity=1},
		complete/.style = {opacity=.35}
	]
	\newcommand{\smallspace}{.07}
	
	\node[vertex] (t) at (0,2){$t$};
	
	\foreach \block/\shift in {a/0 , b/-1cm, g/-2cm }{
		\begin{scope}[name prefix=\block, yshift=\shift]
			\foreach \i/\j in {1/-3,2/-2, 3/-1, 4/1, 5/2}{
				\node[vertex] (\i ) at (\j,0) {};
			} 
			\node[anchor=west] (dots1) at ($(3.east)+(\smallspace,0)$) {\dots};
			\node[anchor=west] (dots2) at ($(5.east)+(\smallspace,0)$) {\dots};	
		\end{scope}
	}
	
	\foreach \block/\name in {a/{$S$ (top)}, b/{$S$ (bottom)}}{
		\foreach \i/\j/\angle  in {
					1/2/10, 
					1/3/15,
					1/4/17,
					1/5/18,
					2/3/10,
					2/4/14,
					2/5/17,
					3/4/12,
					3/5/14,
					4/5/10 }
		{
			\draw[complete] (\block\i) to[out=360-\angle  , in=180+\angle ]  (\block\j);
		}
		
		\draw[region ] ($(\block 1.north west)+(-\smallspace ,\smallspace)$) 
					-- 
					($(\block dots2.north east)!(\block 1.north west)!(\block dots2.south east)  + (\smallspace, \smallspace)$)
					--++ (0,-.7) 
						node[label, pos=.5, anchor= west] (name) {\name}
					-| ($(\block 1.south west)  + (-\smallspace, -\smallspace)$)
					-- cycle;
	}

	\foreach \j in {1,...,5}{
		\draw  (t) to (a\j);
		\draw  (a\j) to (b\j);
		\draw  (b\j) to (g\j); 
		
		\draw (g\j ) --++ (-.47, -.72) --++(.94,0)
				node[label, pos=.5, anchor=south, font=\large, yshift=-.9mm] {$R\displaystyle{_{s_{\text{{\fontsize{4}{1.1}\selectfont \j}}}}}$ }
			--(g\j);
		
		\node[font=\fontsize{6}{7}\selectfont ] at (b\j) {$s_{\text{{\fontsize{3}{4}\selectfont \j}}}$};
	}
	
	\foreach \j in {1,...,3}{
		\node[vertex, minimum size=1.5mm] (c\j) at ($(a\j)+(-.26,.34)$) {};
		\draw (c\j) to (a\j);
	}
	\draw[region] ($(c1.north west)+(-\smallspace ,\smallspace)$) 
			-- ($(c3.north east) + (\smallspace, \smallspace)$)
			-- ($(c3.south east)+(\smallspace, -\smallspace)$) 
			-- ($(c1.south west)+(-\smallspace, -\smallspace)$) 
			-- cycle
				node[label, pos=.5, anchor= east] (cname) {$x$};

\end{tikzpicture}
\end{center}
Under ZF, $V_0=\varnothing$, so there is a function from $V_0$ into non-isomorphic RCPGs.  By applying lemma \ref{inductively_building_labels} inductively, we can build a unique RCPG for each set in $V$, and thus have Frucht's theorem by lemma \ref{enough_labels_implies_fructh}.

\section{Frucht's theorem in Permutation Models}\label{zfa}
\subsection{The Basic Theory of Permutation Models}
The reader completely unfamiliar with permutation models is referred to \cite[chapter 4]{jech} for a better introduction.  This mostly serves to establish my notation.

ZFA set theory enriches ZF set theory with a set\footnote{Sometimes one allows a proper class of atoms; we do not here} $A$ of ``atoms"; sets with no elements, but distinct from the empty set (so, the signature is now $(\in, A)$, and extensionality is modified to exclude elements of $A$).  The notion of rank still makes sense for ZFA, but we have $V_0=\{\varnothing\}\cup A$ instead of just $\{\varnothing\}$.

Unlike ZF models, ZFA models can have non-trivial automorphisms.  Suppose $\pi$ is a bijection on $A$; then $\pi$ acts on $V_0$, and we can inductively define $\pi[S] = \{\pi[x] | x\in S\}$ to define an action of $\pi$ on $V$.  We denote this action with square brackets, to remove ambiguity with function application.

\textbf{Permutation models} are sub-models of a model of ZFA consisting of ``hereditarily symmetric" sets.  To construct one we need three things:
\begin{itemize}
    \item A model $M$ of ZFA
    \item A subgroup $\Sigma$ in $M$ of bijections of $A$ (which we will call ``Automorphisms" of $A$)
    \item A normal filter $\calF$ of subgroups of $\Sigma$.  This means:
    \begin{itemize}
        \item $\calF$ is upwards closed and closed under finite intersections
        \item If $\Psi \in \calF$ and $\sigma \in \Sigma$, then $\sigma \Psi \sigma^{-1}\in \calF$
        \item For each $a\in A$, $\stab(a)\in \calF$
    \end{itemize}
\end{itemize}
Since $\Sigma$ acts on $M$, we can define $\stab(S)$ for any $S\in M$.  We say $S\in M$ is \textbf{symmetric} if $\stab(S)\in \calF$, and \textbf{hereditarily symmetric} if each $x$ in the transitive closure of $S$ is also symmetric.  One can easily show the hereditarily symmetric elements of $M$ form a model of ZFA, with the same set of atoms (see Jech for more details).  This is called a permutation model, and we will denote it $M/\calF$ or $M'$ (when $\calF$ is clear).

We need to differentiate between \textbf{``internal" and ``external" automorphisms} of $A$.  $\sym_M(A)$ is usually much larger than $\sym_{M'}(A)$, so elements of $\Sigma$ can fail to be symmetric.  We call automorphisms of $A$ in $\Sigma \cap M'$ internal, and automorphisms not in $M'$ external.  If $\pi$ is an internal automorphism of $A$, then the mapping $x\mapsto \pi[x]$ is a definable class in $M'$, but this will fail if $\pi$ is not internal.

In most of our examples, $\calF$ will be generated by subgroups of the form $\stab(B)$ for subsets $B\subseteq A$.  In these cases, if $S$ is a symmetric set, there must be some $B\subseteq A$ so $\stab(B)\subseteq \stab(S)$.  $B$ is called a \textbf{support} of $S$ in this case.  In the case $\calF$ is generated by sets $\stab(a)$ for $a\in A$, supports can always be taken as finite subsets of $S$.

\subsection{Counterexamples to Frucht's Theorem in ZFA}\label{counterexamples}
\subsubsection{A general criterion for Frucht's Theorem to fail}

Suppose $G$ is a graph in a permutation model, and $\pi$ is an internal automorphism which fixes $G$.  Since $\pi$ acts on the entire model, $\pi$ acts on $G$ and $\pi$ acts on $\aut(G)$.  It turns out the action of $\pi$ on $\aut(G)$ must be given by conjugation by the action of $\pi$ on $G$.  This fact gives rise to the following theorem\footnote{The same proof actually gives a more technical, slightly stronger theorem (see \cite{details}), but I have yet to find an example where that applies but this does not.}:

\begin{Thm}\label{simplified_no_frucht}
    Let $M'=M/\calF$ be a permutation model given by a filter $\calF$ on $\Sigma$ in $M$. Let $\Gamma$ be a group in $M'$
    
    Suppose that for every $F\in \calF\cap \stab(\Gamma)$, there is $\pi\in F\cap M'$ such that the action of $\pi$ on $\Gamma$ is not an inner automorphism of $\Gamma$
    
    Then $\Gamma$ is not the automorphism group of any graph in $M'$
\end{Thm}
\begin{proof}
    Suppose for contradiction $G\in M'$ is a graph with automorphism group isomorphic to $\Gamma$ in $M'$.  Let $\varphi\in M'$ be an isomorphism from $\Gamma$ to $\aut(G)$ in $M'$.  Since the triple $\langle G, \varphi, \Gamma\rangle$ is in $M'$, it must be symmetric; let $F\subseteq \Sigma$ be its stabilizer in $M$, and $F'=F\cap M'$.

    Since $F$ stabilizes $\Gamma$, by hypothesis there is $\pi\in F'$ which does not act on $\Gamma$ by an inner automorphism; i.e. $\forall x \in \Gamma \exists \gamma \in \Gamma. \pi[\gamma]\neq x \gamma x^{-1}$

    Since $F$ stabilizes $G$, the map $x\mapsto \pi[x]$ defines a graph automorphism $p\colon G\to G$.  Let $\rho=\varphi^{-1}(p)$ be the corresponding element of $\Gamma$.  Since $\pi$ is not inner, we can pick $\gamma\in \Gamma$ so $\pi[\gamma]\neq \rho \gamma \rho^-1$.  Let $g=\varphi(\gamma)$ be the corresponding automorphism of $G$
    
    Since $F$ stabilizes $\varphi$, 
    $$
    \pi[g] = \pi[\varphi(\gamma)] = \pi[\varphi](\pi[\gamma]) = \varphi(\pi[\gamma])
    $$
    However, for fixed $g\in \aut(G)$ and vertex $v\in G$,
    $$p(g(v)) = \pi[g(v)] = \pi[g](\pi[v]) = \pi[g](p(v))$$
    so $\pi[g]=pgp^{-1}$, so
    $$
    \pi[g] =pgp^{-1}= \varphi(\rho) \varphi(\gamma)\varphi(\rho)^{-1} = \varphi(\rho \gamma \rho^{-1})
    $$
    So, $\varphi(\pi[\gamma]) = \varphi(\rho \gamma \rho^{-1})$.  Since $\varphi$ is an isomorphism, this means $\pi[\gamma]=\rho \gamma \rho^{-1}$, contradicting our assumption.
\end{proof}

If we only care about whether Frucht's Theorem is true and not what $\Gamma$ is, we can take $\Gamma$ as the free abelian group on $A$.  Since $\Sigma$ acts faithfully on $A$, it acts faithfully on $\Gamma$, and since $A$ is abelian, every automorphism is outer; so any non-trivial $\pi$ will work in Theorem \ref{simplified_no_frucht}

\begin{Cor}\label{easy_no_frucht}
Suppose $M' = M/\calF$ and every $F\in \calF$ contains a non-trivial internal automorphism.  Then Frucht's Theorem fails in $M'$
\end{Cor}

\subsubsection{Some explicit counterexamples}\label{counterexample}

The easiest interesting example of a permutation model is the Frankel model.  Begin with a model $M$ of ZFA such that $M\models \text{``$A$ is countably infinite"}$.  Take $\Sigma=\sym(A)$ and $\calF$ to be the filter generated by the subgroups $\stab(a)$ for $a\in A$. Then $M'=M/\calF$ is called the Frankel model.

For any $F\in \calF$ we can find $a,b\in A$ so the transposition which swaps $a$ and $b$ is in $F\cap M'$.  Hence we can apply Corollary \ref{easy_no_frucht} to see Frucht's Theorem fails in the Frankel model.  Moreover, Theorem \ref{simplified_no_frucht} gives us several other groups in the Frankel model which are not automorphism groups of graphs:
\begin{enumerate}
    \item Any $A$-fold direct sum or product of a fixed group
    \item\label{free_group} The non-abelian free group on $A$ (swapping two generators is an outer automorphism)
    \item The commutator subgroup of the free group on $A$ (for the same reason as above).  Notably, this group is not free (since no symmetric set freely generates it; see \cite[theorem 10.12]{jech})
    \item The group of alternating permutations of $A$
\end{enumerate}

The second Frankel model, in \cite[section 4.4]{jech}, shows it is possible that a countable set of pairs can have no choice function. We will show Frucht's theorem also fails in this model.

To construct the model, let $M$ be a model where $A=\{a_0, a_1, a_2, a_3, \dots\}$, and $\Sigma$ the group of permutations of $A$ setwise preserving each of $\{a_0, a_1\}, \{a_2,a_3\}, \dots$ (so $\Sigma\cong \prod\limits_{n\in \NN} C_2$).  Let $\calF$ be the filter on $\Sigma$ generated by subgroups $\stab(a_i)$.  The second Frankel model is $M/\calF$.

In this example, $\Sigma$ only contains iternal automorphisms of $A$. To see this, notice that, for any $\sigma, \pi\in \Sigma$, $\sigma[\pi] = 
\sigma \pi \sigma^{-1}$, so the stabilizer of $\pi$ as an element of $M$ is precisely it's centralizer in $\Sigma$.  Since $\Sigma$ is abelian, this is $\Sigma$.  From Corollary \ref{easy_no_frucht}, it easily follows that Frucht's Theorem fails in $M'$

\subsection{Some ZFA examples where Frucht's Theorem holds}\label{non-counterexamples}
There are various trival examples of ZFA models where Frucht's Theorem holds.  For instance, any ZFA model where AC holds will satisfy Frucht's theorem, since any colored graph can by colored by ordinals.  Also, any ZF model can be regarded as a ZFA model, where $A=\varnothing$.

Frucht's theorem can also hold in ZFA for non-trivial reasons: 
\begin{Thm}\label{labels_implies_frucht_ZFA}
    Suppose $M\models \text{``there is a function from $A$ into non-isomorphic RCPGs"}$.  Then $M\models\text{Frucht's theorem}$.
\end{Thm}
This follows from the same proof of Frucht's theorem in section \ref{zf}; we apply lemma \ref{inductively_building_labels} inductively to construct an RCPG for every set, then conclude Frucht's Theorem by lemma \ref{enough_labels_implies_fructh}.

\subsubsection{The random graph model}
Let $M$ be a model of ZFA where $A$ is countable and $E$ be a binary relation on $A$ so $\langle A,E\rangle$ is (isomorphic to) the countable random graph (which we will simply call $A$).  Let $\Sigma$ be the group of graph automorphisms of $A$.  Again, we will take $\calF$ to be the filter generated by $\stab(a)$ for $a\in A$.  We will call $M'=M/\calF$ the random graph model.

We want to show that $M'$ satisfies Frucht's Theorem, using theorem \ref{labels_implies_frucht_ZFA}.  We will show the neighborhood of each $a\in A$ is a RCPG, and no two atoms have isomorphic neighborhoods.  To develop an intuition why this is so, we will first show the atoms form a rigid graph in $M$.

\begin{Lem}\label{once_lemma_10}
$A$ has no internal automorphisms in $M'$
\end{Lem}
We should remark the set of bijections on $A$ is not empty in $M'$; but no bijections on $A$ in $M'$ preserve the graph structure.  The proof of Lemma \ref{once_lemma_10} makes use of the following elementary result, whose proof we leave for the Appendix:
\begin{restatable}{Lem}{randgraphlemma}
\label{rigidrandgraph}
Let $f$ be a non-trivial automorphism of a countable random graph, and $E$ be a finite subset of the countable random graph.  There is an automorphism $\pi$ of the random graph which does not commute with $f$ and fixes $E$
pointwise.

Moreover, if $f(e) = e$ for some $e\in E$, we can chose $\pi$ so there is a vertex $v$ adjacent to $e$ such that $\pi f(v)\neq f\pi(v)$
\end{restatable}

\begin{proof}[Proof of Lemma \ref{once_lemma_10}]
For the sake of contradiction, suppose $f$ is a non-trivial internal graph automorphism of $A$ in $M'$, and let $E$ be a support of $f$.  Let $\pi$ be an automorphism of $A$ as in Lemma \ref{rigidrandgraph} which fixes $E$ and does not commute with $f$.

Since $\pi$ fixes $E$, we must have $\pi[f]=f$.  Notice the action of $\pi$ on $\aut(A)$ in $M'$ is precisely by conjugation; $f(a) = b$ iff $(\pi [f])(\pi[a]) = \pi[b]$, so $\pi[f] = \pi f \pi^{-1}$.  However, $\pi$ and $f$ do not commute, so this is a contradiction.
\end{proof}

For each $a\in A$, let $G_a$ denote the neighborhood of $a$ in $A$, an induced subgraph of $A$ with $a$ as the distinguished vertex.

\begin{Lem}
Assigning $a$ to $G_a$ is a function from $A$ into pairwise non-isomorphic RCPGs; that is,
\begin{enumerate}
    \item If $a,b\in A$ are distinct atoms, $G_a$ and $G_b$ are not isomorphic
    \item Each $G_a$ is rigid
    \item Each $G_a$ is connected
\end{enumerate}
\end{Lem}
\begin{proof}
\begin{enumerate}
    \item Suppose $G_a$ and $G_b$ are isomorphic in $M'$; then there is some bijective function $f$ in $M'$ from $G_a$ to $G_b$.  Let $E$ be a support of $f$.  Pick $x\in A\setminus E$ such that $x$ is adjacent to $a$ and not $b$ (so $x\in G_a\setminus G_b$).  Let $y$ be a vertex in $G_b$ such that $y\neq f(x)$, but $y$ and $f(x)$ are adjacent to the same verticies in $E\cup\{a,b,x\}$.  Since automorphisms of finite subsets extend to a countable random graph, there is (externally) an automorphism $\pi$ which fixes $E\cup\{a,b,x\}$ and swaps $f(x)$ and $y$.  Since $\pi$ fixes $E$, $\pi[f]=f$.  However, $\pi(f(x)) = y \neq f(x) = f(\pi(x))$, so $\pi[f]\neq f$.
    
    \item Suppose $f$ is an automorphism of $G_a$, and let $E$ be a support of $f$.  Without loss of generality, suppose $a\in E$.  By our stronger clause in lemma \ref{rigidrandgraph}, we can construct $\pi$ which stabilizes $E$ so that $\pi(f(\pi^{-1}(x)))\neq f(x)$ for some $x\in G_a$ (this composition makes sense because $\pi$ fixes $a$, so $\pi^{-1}(x)\in G_a$).  However, $\pi$ fixes $E$, so we should have $\pi[f]=f$.  So, $G_a$ is rigid.
    \item Every vertex in $G_a$ is attached to $a$, so $G_a$ is connected.
\end{enumerate}
\end{proof}

By theorem \ref{labels_implies_frucht_ZFA}, Frucht's theorem holds in the random graph model.

\subsubsection{The Ordered Mostowski Model}\label{ordered-mostowski}
While I do not know of literature references for the random graph model, this model is well studied.  See \cite[section 4.5]{jech}.

Begin with a model $M$ where $A$ is countable and $\leq$ is a dense linear order without endpoints on $A$.  Let $\Sigma$ be the group of order preserving automorphisms of $A$ in $M$, and $\calF$ be the filter generated by subgroups $\stab(a)$ for $a\in A$.  The ordered Mostowki model $M'$ is the permutation model $M/\calF$.

As above, we will use theorem \ref{labels_implies_frucht_ZFA}.  First, we will show that for each $a\in A$, $\pred(a)=\{x\in A | x\leq a\}$ is a rigid (in $M'$) linear order, and uniquely determines $a$. Then, we encode these linear orders as graphs with the same automorphism group.

\begin{Lem}\label{orders_rigid_in_mostowski}
In $M'$, $\langle A,\leq\rangle$ is a rigid linear order.

Moreover, for each $a\in A$, the initial segment $\pred(a)$ of $A$ is also a rigid linear order, and no two initial segments are isomorphic.
\end{Lem}
\begin{proof}
Suppose $f\colon A\to A$ is a non-identity order automorphism in $M'$.  Let $E\subseteq A$ be a support of $f$.  


Clearly $\{x | f(x)\neq x\}$ is infinite, so since $E$ is finite, we can pick $x\in A$ so that $f(x)\neq x$ and both $x, f(x)\notin E$.  Let $\pi$ be an order automorphism of $A$ in $M$ which fixes $E$ so $\pi(x)>x$ and $\pi(f(x))<f(x)$ (in $M$, $A$ is isomorphic to $\mathbb{Q}$, where this is clearly possible).  

Since $\pi$ fixes $E$, we must have $\pi[f]=f$.  We know $\pi$ acts by conjugation, so $\pi[f](\pi(x))) = \pi (f(\pi^{-1}(\pi(x)))) = \pi(f(x))<f(x)$.  However, $\pi(x)>x$, so $f(\pi(x))>f(x)$.  So, $\pi(f)(\pi(x))\neq f(\pi(x))$, so $\pi[f]\neq f$.  This is a contradiction, so $A$ is rigid in $M'$.

Each $\pred(a)$ is rigid in $M'$ because if $g$ is an order automorphism of $\pred(a)$, we can extend it to an automorphism of all of $A$ by:
$$
x\mapsto\begin{cases}
g(x) & x\leq a\\
x & x>a
\end{cases}
$$

To see no two $\pred(a)$'s are isomorphic, we remark that our argument $A$ is rigid doesn't depend on surjectivity of $f$.  So, there is no non-identity injective order preserving function from $A$ to $A$.  However, suppose $g$ were an order isomorphism between $\pred(a)$ and $\pred(b)$ for some $a<b$.  Then:
$$
x\mapsto\begin{cases}
g(x) & x\leq b\\
x & x>b
\end{cases}
$$
is an order preserving map from $A$ to $A\setminus (a,b)\subsetneq A$; but no such injection exists.

\end{proof}

Next, we need to encode linear orders in graph.  In fact, one can encode arbitrary first order structures (at least with finite signatures) in graphs; one can view this as a very special case of the representability machinery in \cite[chapter 5]{hodges} (or, for a clearer construction, see \cite[section 4.3]{simon}).  

Given the machinery in section \ref{zf} to convert directed graphs into graphs, we can use a much simpler construction:

\begin{Def}
Let $\langle L,\leq\rangle$ be a linear order. Regard $\langle L, <\rangle$ as a directed graph.  Let $\Gamma_L$ be the (non-directed) graph constructed from $\langle L,<\rangle$ as described in in section \ref{zf}.
\end{Def}

\begin{restatable}{Lem}{orderToGraph}
\label{ordercoding}
If $L$ and $L'$ are linear orders
\begin{enumerate}
    \item $L$ and $L'$ are isomorphic orders iff $\Gamma_L$ and $\Gamma_L'$ are isomorphic graphs
    \item $\Gamma_L$ is a rigid graph iff $L$ is a rigid order.
    \item $\Gamma_L$ is connected
\end{enumerate}
\end{restatable}
We  will not prove this, as it is straightforward, but we remark that any encoding of orders in graphs with these properties would be sufficient.

Consider the assignment sending each $a\in A$ to the graph $G_a=\Gamma_{\pred(a)}$, with $a$ as a distinguished vertex.  By lemmas \ref{orders_rigid_in_mostowski} and \ref{ordercoding}, this defines a function from $A$ into non-isomorphic RCPGs.  So, Frucht's theorem holds for the ordered Mostowski model by theorem \ref{labels_implies_frucht_ZFA}

\section{A variations of Frucht's theorem that can fail in ZF}\label{babai}

In \cite{BABAI197826}, Babai tries to minimize the number vertex orbits in a graph with a given automorphism group.  Babai's main theorem says, under AC, every group is the automorphism group of a directed graph with 1 vertex orbit (or, by an easy corollary, an undirected graph with 3 vertex orbits).  Unlike Frucht's theorem, which can be witnessed by graphs of arbitrary rank, a witness for Babai's theorem for a group $\Gamma$ can always be witnessed by a graph with vertex set contained in $\calP(\Gamma)$ (specifically, the cosets of a vertex stabilizer).

Suppose $\Gamma$ is a group where Frucht's theorem fails in some permutation model $M$, and $\Gamma\in V_\alpha^M$.  Then $V_{\alpha+2}^M\models \text{Frucht's theorem fails for $\Gamma$}$.  By Jech-Sochor embedding, we can construct a symmetric model of ZF $N$ with a group $\Gamma'$ so  $\calP^2(\Gamma')^N\models \text{Frucht's theorem fails for $\Gamma'$}$.  This means Babai's theorem must fail for $\Gamma'$ in $N$.  A similar argument shows the following:
\begin{Thm}
    There is a model of ZF with a group which is not the automorphism group of a graph with finitely many vertex orbits.
\end{Thm}
It is difficult to make Babai's theorem fail much worse than this without choice; since not all cardinalities are comparable, it may be impossible to ``minimize" the number of vertex orbits in a graph with a fixed automorphism group.

\section{Further Questions}
Most examples of permutation models I know either have many non-trivial internal automorphisms (so Frucht's Theorem fails by Corollary \ref{easy_no_frucht}), or their atoms have some kind of first-order structure, which can be coded into a graph, and Frucht's Theorem holds.  So, every example I have considered supports the following conjecture, the converse of Corollary \ref{easy_no_frucht}:
\begin{Conj}\label{no_frucht_converse}
If the Axiom of Choice holds in $M$ and $M'=M/\calF$ is a permutation model inside $M$ where Frucht's Theorem fails, then every $F\in \calF$ contains a non-trivial internal automorphism.
\end{Conj}
The assumption that the Axiom of Choice holds in $M$ may be unnecessarily strong.  At a minimum, we need to assume Frucht's Theorem holds in $M$, to avoid the trivial case $M'=M$ (and $\{1\}\in \calF$).

As in \cite{degroot} and \cite{sabidussi}, the Axiom of Choice implies Frucht's Theorem, even over ZFA.  This might cause one to wonder whether something weaker than the Axiom of Choice implies Frucht's Theorem.  
Several of these questions are easily answered by this paper, but there are some interesting open ones (perhaps most notably whether the Ordering Principal implies Frucht's Theorem).
I summarize results I do and do not know about this in Appendix \ref{weak_choice_and_frucht}.  

It unknown whether Babai's Theorem is consistent with the failure of choice.  I expect it holds in the random graph and ordered mostowski models, but it is difficult to prove in general.  As one example, if $\Gamma$ is the free abelian group on $A$ in the random graph model, and $S$ is the set of elements of $\Gamma$ with components 0 or 1, with the 1's occurring on a clique in $A$, then $\Gamma$ is the automorphism group of the directed graph $\cay(\Gamma,S)$.

\section{Appendix}
\subsection{Comparison of Frucht's Theorem with Weak Versions of the Axiom of Choice}\label{weak_choice_and_frucht}
Here are the implications I know between Frucht's Theorem and classical weak versions of Choice.

The following statements imply Frucht's Theorem over ZFA
\begin{enumerate}
    \item The Axiom of Choice
    \item The (Kinna Wagner) Selection Principal, which is equivalent to the statement ``for every $S$, there is an ordinal $\alpha$ and an injection $S\to \calP(\alpha)$".  Subsets of ordinals can be coded by graphs following section \ref{zf}
    \item For every $S\in M$, there is $S'\in \mathrm{Ker}(M)$ and and injection $S\to S'$ ($\mathrm{Ker}(M)=\calP^{\mathrm{Ord}}(\varnothing)$, the largest inner model of ZF) \\
    The reasoning is similar.
\end{enumerate}

The following statements are consistent with the failure of Frucht's Theorem.  All of them follow by noting Corollary \ref{easy_no_frucht} applies to the corresponding permutation model given in \cite{jech}. 
\begin{enumerate}
    \item For any fixed finite $n$, if $S$ is a set of $n$-element sets, $S$ has a choice function
    \item For every finite $n$, if $S$ is a set of $n$-element sets, $S$ has a choice function.\\
    Note that this is properly weaker than choice from sets of finite sets
    \item The Axiom of Multiple Choice (If $S$ is a set of non-empty sets, there is a function $f$ on $S$ so each $f(x)$ is a non-empty finite subset of $x$)\\
    It should be noted this is equivalent to AC over ZF
    \item For any fixed ordinal $\lambda$, $AC_\lambda$: if $S$ is a set equinumerous with some ordinal $\lambda$, then $S$ has a choice function
    \item For every ordinal $\lambda$, $AC_\lambda$\\
    (This is weaker than full AC because non-well-orderable sets may not have choice functions)
    \item For any fixed ordinal $\lambda$, Dependent Choice of length $\lambda$
    \item For any fixed ordinal $\lambda$, $W_\lambda$: for every $S$ there is an injection from $\lambda$ to $S$ or an injection from $S$ to $\lambda$
    \item Every set has a non-trivial ultrafilter\\
    This is weaker than the Prime Ideal Theorem
    \item The Hahn Banach Theorem
\end{enumerate}

It is open whether the following statements imply Frucht's Theorem:
\begin{enumerate}
    \item Prime Ideal Theorem
    \item Ordering Principal
    \item Order Extension Principal (every partial order extends to a linear order)
    \item Axiom of Choice for collections of well-orderable sets
    \item Axiom of Choice for collections of finite sets
\end{enumerate}
In all of the permutation models I know of where these hold, the atoms have some homogeneous structure, which becomes rigid in the permutation model.  So, the technique from section \ref{ordered-mostowski} can show Frucht's Theorem holds in these models.  Answering these questions would require either a clever new permutation model, or a very different proof of Frucht's Theorem, both of which would be interesting.

Lastly, I should remark that Frucht's Theorem does not imply any sort of choice, as Frucht's Theorem will hold in any ZF model (which are trivial ZFA models as well), no matter how badly Choice fails.

\subsection{Proof of lemma ~\ref{rigidrandgraph}}
\randgraphlemma*

Suppose there is an $e\in E$ so $f(e)\notin E$.  Let $\pi$ be any automorphism of the random graph so $\pi$ fixes $E$ and $\pi(f(e))\neq f(e)$ (which we can construct by extending an automorphism of a finite subgraph).  Then $\pi(e)=e$, so $f(\pi(e))=f(e)\neq \pi(f(e))$.

So, we can suppose $f$ fixes $E$ setwise.  Since $f$ is a non-trivial automorphism of the random graph, $\{v | f(v)\neq v\}$ is infinite (this is an easy, well known argument).  Since $E$ is finite, there is $v\notin E$ so $f(v)\neq v$.  Notice $f(v)\notin E$ as well, since $f$ is bijective and fixes $E$.  Let $\pi$ be an automorphism of the random graph which fixes $E\cup \{v\}$ so $\pi(f(v))\neq f(v)$.  Since $\pi(v)=v$, $f(\pi(v))=f(v)\neq \pi(f(v))$.  So, we can always fine $\pi$ so $f$ and $\pi$ don't commute.

Moreover, suppose there is some $e\in E$ so $f(e)=e$.  Let $G_e$ be the set of verticies adjacent to $e$.  If there is $e'\in E\cap G_e$ so $f(e')\notin E$, we can let $\pi$ be an automorphism which fixes $E$ but moves $f(e')$, and $\pi(f(e'))\neq f(\pi(e'))$ as above.  

Otherwise, $f$ fixes $E\cap G_e$ setwise. Since $f(e)=e$, $f\restr G_e$ is an automorphism of the random graph $G_e$, so $\{v\in G_e|f(v)\neq v\}$ is infinite, and $E\cap G_e$ is finite, so we can find $v\in G_e\setminus E$ so $f(v)\neq v$.  Notice $f(v)\notin E\cap G_e$ since $f$ is bijective and fixes $E\cap G_e$, and $f(v)\notin E\setminus G_e$ since $f(v)$ is adjacent to $f(e)=e$, so $f(v)\notin E$.  So, let $\pi$ be an automorphism of the random graph which fixes $E\cup \{v\}$ and moves $f(v)$.  As above, $f(\pi(v))\neq \pi(f(v))$, and $v\in G_e$, as required.

\printbibliography

\end{document}